\documentclass[12pt]{article}
\usepackage{graphicx}
\usepackage{setspace}
\usepackage{amsmath, amssymb,amsthm, amsfonts, latexsym, mathrsfs}
\usepackage[mathscr]{eucal}

\usepackage{color}

\usepackage{framed}
\usepackage{authblk}

\usepackage[titletoc]{appendix}

\newtheorem{thm}{Theorem}[section]
\newtheorem{lma}{Lemma}[section]
\newtheorem{prop}{Proposition}[section]
\newtheorem{cor}{Corollary}[section]

\theoremstyle{definition}

\theoremstyle{remark}
\newtheorem{remark}{Remark}[section]

\numberwithin{equation}{section}

\newcommand{\tr}{\mbox{tr}}

\newcommand{\Ric}{\mbox{Ric}}
\newcommand{\R}{\mathbb R}

\newcommand{\be}{\begin{equation}}
\newcommand{\ee}{\end{equation}}
\newcommand{\bee}{\begin{equation*}}
\newcommand{\eee}{\end{equation*}}

\def\p{\partial}

\def\la{\langle}
\def\ra{\rangle}
\def\lf{\left}
\def\ri{\right}
\def\Pi{\displaystyle{\mathbb{II}}}

\def\dist{\text{dist}}

\def\S{\Sigma}

\renewcommand\L{\triangle}
\newcommand\D{\nabla}
\newcommand\ric{{\rm Ric}}
\renewcommand\d{\partial}
\newcommand\g{\gamma}

\newcommand\bbR{{\mathbb R}}

\renewcommand\d{\partial}

\renewcommand\L{\triangle}

\newcommand\beq{\begin{equation}}
\newcommand\eeq{\end{equation}}
\newcommand\ben{\begin{enumerate}}
\newcommand\een{\end{enumerate}}
\newcommand\bit{\begin{itemize}}
\newcommand\eit{\end{itemize}}

\newcounter{mnotecount}[section]

\title{Variational and rigidity properties of static potentials}
\author{Gregory J. Galloway}%\thanks{Work partially supported by NSF grant DMS-0708048.} }
\author{Pengzi Miao}
\affil{Department of Mathematics\\ University of Miami }

\begin{document}
\date{}
\maketitle

\begin{abstract}  
In this paper we study some global properties of static potentials on  asymptotically flat $3$-manifolds $(M,g)$ in the nonvacuum setting.   Heuristically, a static potential $f$ represents the (signed) length along $M$ of an irrotational timelike Killing vector field, which can degenerate on surfaces corresponding to the zero set of $f$.  Assuming a suitable version of the null energy condition, we prove that a noncompact component of the zero set must be area minimizing.  From this we obtain some rigidity results for static potentials that have noncompact zero set components, or equivalently, that are unbounded.  Roughly speaking, these results show, at the pure initial data level, that `boost-type' Killing vector fields can exist only under special circumstances.  
\end{abstract}

\section{Introduction}  

Consider a static spacetime
\beq \label{eq-static}
\bar M = \bbR \times M \,, \quad \bar g = -f^2 dt^2 + g ,
%\,.
\eeq
where $ (M, g)$ is a Riemannian $3$-manifold and $f$ is a positive function on $M$.
The function $f$  and the Ricci tensors $\ric_{\bar {M}}$, and $ \ric $, of $(\bar M, \bar g)$, and $(M,g)$, respectively,  are related by,
\begin{align}
\D^2 f &= f (\ric   -  \ric_{\bar M} |_{M} )  \label{ein1}  \,,\\
\L f &=  \ric_{\bar M}(u,u)  f \label{ein2} \, ,
\end{align}
where 
$ \ric_{\bar M}  |_{M} $ denotes the restriction of $ \ric_{\bar M}  $ to the tangent space of $M$ and
$ u = f^{-1} \p_t $ is the future timelike unit normal to $M$ in $(\bar{M}, \bar{g})$. 
When $(\bar M, \bar g)$ satisfies the Einstein equation, then Equations \eqref{ein1} and \eqref{ein2}  are the (in general nonvacuum) Einstein field equations for a static spacetime.  

Let $ R_{\bar M}$ and $ R$ be the scalar curvature of $(\bar M, \bar g)$ and $(M,g)$, respectively.
In terms of the Einstein tensor,
\beq
G = \ric_{\bar M} -\frac12 R_{\bar M} \, \bar {g}   \, ,
\eeq
Equations  \eqref{ein1} and \eqref{ein2}  become,
\begin{align}
\D^2 f &= f [\ric -  \g + \frac12(\tr \g - \rho)g]  \,,  \label{ein3}\\
\L f &=  \frac12(\rho + \tr \g)   f  \,,\label{ein4}
\end{align}
where $\rho = G(u,u) =  \frac12 R $,  
and $\g$ is $G$ restricted to $TM$.  If one assumes the Einstein equation holds:  $G =T$,  where $T$ is the energy-momentum tensor, then decay conditions on $\rho$ and $\gamma$ may be viewed as decay conditions on $T$.

More generally, for a given Riemannian $3$-manifold $(M,g)$, scalar field $\rho$ and symmetric $2$-tensor $\g$,  a {nontrivial} smooth function $f$, without any sign assumptions, will be called a {\it static potential}  
for  $(M,g, \rho,\gamma)$ provided equations \eqref{ein3} and \eqref{ein4} hold.  Static potentials in the {\it vacuum} case ($\rho = 0$, $\g = 0$) arose in the work of Corvino \cite{Corvino}, where they correspond to nontrivial elements in the kernel of the adjoint of the linearized scalar curvature map (see also \cite{Fischer}).
From a slightly different point of view, in studying static potentials, we are in essence  considering Killing initial data \cite{Beig-Chrusciel-97} with zero shift.
 
In this paper we establish some global properties of static potentials for asymptotically flat $3$-manifolds $(M,g)$, subject to natural energy and decay conditions on $\rho$ and $\gamma$.  
In Section 2 we present some local and asymptotic properties of static potentials. 
In Section 3 we establish a basic restriction on the occurrence of minimal surfaces having boundaries which lie on the zero set of a static potential; cf. Theorem \ref{non-minimal}. This result is used in Section 4 to show that a noncompact component of the zero set must be area minimizing; cf. Theorem \ref{thm-minimizing}.  From this we obtain some nonexistence and rigidity results for static potentials 
that have noncompact zero set components, or equivalently, that are unbounded.  Roughly speaking, these results show, at the pure initial data level, that `boost-type' Killing vector fields can exist only under special circumstances.  See \cite[Theorem 1.1]{Beig-Chrusciel-97b} for a related spacetime result, which applies to asymptotically flat spacetimes that admit boost domains.

\section{Preliminaries} \label{sec-p}
 
 In this section,  we 
 collect
 some preliminary  
 results concerning a nontrivial solution $f$ to 
 \be \label{ein3-1}
\D^2 f = f [\ric -  \g + \frac12(\tr \g - \rho)g]  
 \ee
on $(M, g, \gamma, \rho)$, where $M$ is always assumed to be  connected. 
 
We start  with some local  properties   of the zero set of $f$. 
The following lemma is an analogue of   \cite[Lemma 2.1]{static-af} (which focused  on the case $ \rho = 0 $ and $ \gamma =0$).

\begin{lma}\label{lma-static-basic} 
Suppose $f$ is a nontrivial solution to \eqref{ein3-1}. 
  Let $ \S =f^{-1}(0)$. Suppose $ \Sigma \neq \emptyset $. 
\begin{enumerate}
  \item [(i)] $\S$ is a totally geodesic hypersurface and
   $|\nabla f|$ is a positive constant on each connected component of $ \S$.
  \item [(ii)]  $\nabla f$ is an eigenvector of  $\Ric$ along $ \Sigma$. 
\item [(iii)]  Suppose $ \rho = 0$ and $ \gamma = 0 $  along $ \S$.
At  any  $p \in \S$, let $\{ e_1,e_2, e_3 \}$ be an  orthonormal frame
 such that   $e_3 $ is normal to $\S$.
 Let $ R_{ijkl}$ denote the component of the curvature tensor in this frame such that 
 $ R_{ijij} $ equals the sectional curvature of the tangent $2$-plane spanned by $ e_i$ and $e_j$ for $ i \neq j \in \{ 1, 2, 3 \}$. 
Then 
$$  R_{1313} = R_{2323} = - \frac12 R_{1212} . $$
As a result,  $ R = 0 $ along $ \Sigma$ where $ R$ is the scalar curvature 
$(M, g)$
and 
\be
\Ric |_{T \Sigma} = \frac12 K g |_{ T \Sigma} , \ \ \Ric(e_3, e_3) = - K ,
\ee
where $ \Ric |_{T \Sigma}$, $ g|_{T\Sigma}$ denote the restriction of $ \Ric $, $g$  to the tangent space to $\Sigma$ respectively, and $K$ is the 
 the Gaussian curvature of $ \Sigma$. 
\end{enumerate}

\end{lma}
\begin{proof} (i)
Let   $ p \in \S$. If $ \nabla f(p)=0$, then along any geodesic $\beta(t)$ emanating from $p$, $f(t): = f(\beta(t))$ satisfies 
$$f''(t)= h(t) f(t), \ f'(0) = 0 , \ f(0) = 0 , $$
where
$ h(t)  = \Ric (\beta'(t), \beta'(t) ) - \gamma (\beta'(t), \beta'(t) ) + \frac12 ( \tr \gamma - \rho )  (\beta(t) ) $.
This implies $f$ is zero near $p$. 
Now consider the set  
$$ \Lambda = \{ x \in M \ | \ f (x) = 0 \ \mathrm{and} \ \ \nabla f(x) = 0 \} .$$  
Clearly, $ \Lambda$ is a closed set. The above argument shows that $ \Lambda $ is also  open. 
Since $ p \in \Lambda$ and $M$ is connected, we conclude $\Lambda = M$, contradicting  the fact that $f$ is nontrivial. 
Therefore $ \nabla f (p) \neq 0$, $ \forall \ p \in \Sigma$, 
which implies   $\S$ is an embedded surface. Along $ \Sigma$, 
 \eqref{ein3} shows   $ \nabla^2 f (X, Y) $ = 0  and $ \nabla^2 f (X, \nabla f ) = 0 $ for any tangent vectors  $X, Y$ tangential to $ \S$.
This  readily implies that  $ \S$ is totally geodesic and  $| \nabla f |^2  $ is a constant along $ \Sigma$.

(ii)  Since  $\S$ is totally geodesic,  it follows from the Codazzi equation that   
\be \label{eq-Ric-cross}
\Ric(\nu,X)=0
\ee 
for all $X$ tangent to $ \S$, where $\nu$ is the unit normal of $\S$.
Therefore, $ \nabla f = \frac{\p f}{\p \nu} \nu $ is an eigenvector of $ \Ric$.

(iii)    Let   $ R_{ij } = \Ric (e_i, e_j)$, where $ i, j \in  \{ 1, 2, 3 \}$. 
 Differentiating \eqref{ein3-1}
 and applying the fact  $\rho = 0$ and $\gamma = 0 $ along $\Sigma$,    we have 
$ f_{; ijk} = f_{;k} R_{ij}  $. 
Setting $ k = 3 $ gives
 \bee 
 f_{;ij 3} = f_{; 3} R_{ij} . 
 \eee
On the other hand, 
\be \label{eq-d-f-1-2}
f_{;ij3} - f_{;i3j} =  - \sum_{l=1}^3 R_{ l i 3 j} f_{; l} = - R_{3i3j} f_{;3} .
\ee
Therefore,
\be \label{eq-d-f-2}
\begin{split}
f_{;3} R_{\alpha \beta} = & \ f_{; \alpha 3 \beta} - R_{3 \alpha 3 \beta} f_{;3} 
=  - R_{3 \alpha 3 \beta} f_{;3} , 
\end{split} 
\ee
where  $ \alpha, \beta \in \{ 1, 2 \}$ and we used the fact  $ f_{; ij} = 0 $ along $ \Sigma$. 
Since $ f_{;3} \neq  0 $, \eqref{eq-d-f-2}  implies 
\be \label{eq-d-f-3}
R_{\alpha \beta} = - R_{3 \alpha 3 \beta}  .
\ee
By the definition of $R_{\alpha \beta}$, this is equivalent to 
\be \label{eq-d-f-4}
 R_{1\alpha1\beta} + R_{2\alpha 2\beta } + R_{3 \alpha 3 \beta}  = - R_{3 \alpha 3 \beta}. 
\ee
By taking  $ \alpha = \beta = 1 $ and $ 2 $ respectively, \eqref{eq-d-f-4} implies 
\be \label{eq-d-f-5}
R_{3131} = R_{3232} = - \frac12 R_{1212} ,
\ee
where  $ R_{1212} = K $ since $\Sigma$ is totally geodesic. 
This completes the proof. 
\end{proof}

Next, we assume that $M$ is diffeomorphic  to $ \R^3 \setminus B $,
where $B$ is an open  Euclidean ball centered at the origin,
and $g$ is a smooth metric on $ M$ such that
with respect to the standard coordinates $\{x_i \}$ on $ \R^3$,
$g$ satisfies
\be\label{eq-AF-def}
g_{ij}=\delta_{ij}+b_{ij} \  \ \mathrm{with} \ \
b_{ij} = O_2 (|x|^{-\tau})
\ee
for some constant  $\tau \in ( \frac12 , 1]$.  
We also assume the   symmetric $(0,2)$ tensor $\gamma$ 
and the scalar field $ \rho$ satisfy 
\be \label{eq-gamma-decay}
\gamma_{ij} = O ( |x|^{-2 -\tau} ) , \ \ \rho = O ( | x|^{- 2 - \tau} ) .
\ee
 
The following proposition concerning the asymptotic behavior of $f$ near infinity follows
 from \cite[Proposition 2.1]{Beig-Chrusciel-96} (by setting the shift vector $Y$ in  \cite{Beig-Chrusciel-96} equal to zero).

\begin{prop}\label{prop-AF-static-f}
Suppose  $f$ is a solution to \eqref{ein3-1} on  $(M, g, \rho, \gamma)$
satisfying \eqref{eq-AF-def} and \eqref{eq-gamma-decay}. 
Then
\begin{enumerate}
\item[(i)] there exists a   tuple $(a_1, a_2, a_3) \, \in \R^3  $ such that
 \bee \label{eq-f-xi}
  f= a_1 x_1 + a_2 x_2 + a_3 x_3 + h
 \eee
 where  $h$ satisfies $  \p h  = O_1 ( | x |^{ - \tau} )$
and
\bee \label{eq-condition-h-0}
| h | =
\lf\{
\begin{array}{lc}
O( |x|^{1-\tau}) & \  \mathrm{when} \   \tau < 1, \\
 O ( \ln |x| ) & \  \mathrm{when}   \ \tau = 1.
 \end{array}
 \ri.
\eee

\item[(ii)]  
$ (a_1, a_2, a_3) = (0, 0, 0) $ if and only if  $ f $ has a finite limit at $\infty$. 
 In this case, $ \lim_{ x \rightarrow \infty} f =0 $  only if $f$ is  identically zero.
\end{enumerate}

\end{prop}

Proposition \ref{prop-AF-static-f} itself implies that the zero set of an unbounded $f$
 near infinity   has a controlled graphical structure.

\begin{prop}  \label{prop-AF-static-zeroset}
Suppose $f$ is  an 
unbounded 
solution to \eqref{ein3-1} on  $(M, g, \rho, \gamma)$
satisfying \eqref{eq-AF-def} and \eqref{eq-gamma-decay}. 
Then there exists a new set of  coordinates $\{ y_i \}$
near the infinity of $(M,g)$,
obtained by
a rotation of $\{x_i \}$ such that, outside a compact set,
$ f^{-1} (0)  $
is given by the graph of a smooth  function 
$ q = q (\bar {y}) $, where $\bar y = (y_2, y_3) $,  over
$$ \Omega_C = \{ (0, y_2, y_3) \ | \ y_2^2 + y_3^2 > C^2 \}  $$ 
for some constant $ C>0$, where   $ q $ satisfies
\be \label{eq-condition-q}
    \p q   = O_1 (|\bar{y}|^{-\tau} )
\  \ \mathrm{and} \ \
| q | =
\lf\{
\begin{array}{ll}
O( | \bar{y}|^{1-\tau}) & \  \mathrm{when} \   \tau < 1 \\
 O ( \ln |\bar{y}| ) & \  \mathrm{when}   \ \tau = 1.
 \end{array}
 \ri.
\ee
%Here $ \bar{y} = (y_2, y_3) $.
As a result,
if $\gamma_{_R} \subset f^{-1}(0) $ is the curve   given by
$$ \gamma_{_R} = \{ ( q( y_2, y_3) , y_2, y_3) \ | \ y_2^2 + y_3^2 = R^2 \}  $$
and $\kappa $ is the geodesic curvature of $ \gamma_{_R} $ in $f^{-1}(0)$, then
\be  \label{eq-g-curvature}
\lim_{R \rightarrow \infty} \int_{\gamma_{_R}}  \kappa \  d s = 2 \pi .
\ee
\end{prop}

Proposition \ref{prop-AF-static-zeroset} follows   from Proposition \ref{prop-AF-static-f} by the exact same proof  in \cite{static-af} that proves
\cite[Proposition 3.2]{static-af} using \cite[Proposition 3.1]{static-af}. 

\section{Static potentials and minimal surfaces}

In this section, we assume $(M, g)$ is a complete,  asymptotically flat $3$-manifold without boundary, with finitely many ends.
The triple ($g$, $\gamma$, $ \rho $) is assumed to  satisfy the decay assumptions \eqref{eq-AF-def} and \eqref{eq-gamma-decay} 
on  each end.   

We say that $(M, g, \gamma, \rho)$ satisfies the null energy condition (NEC) provided,
\beq\label{eq-NE}
\rho + \g(X,X) \ge 0  \quad \text{for all unit vectors } X.  
\eeq
For the static metric \eqref{eq-static}, this is equivalent to requiring $\ric_{\bar M}(K,K) \ge 0$ for all null vectors $K$ along $M$.

The aim of this section is to establish the following theorem. 

\begin{thm}\label{non-minimal} 
Assume $(M, g, \gamma, \rho)$ satisfies the NEC and $f$ is a  static potential. 
Let $ U$ be an unbounded, connected component of $ \{ f \neq 0 \}$. 
Then there does not exist any compact subset $S \subset M$ such that 
$ S \setminus \p U $ is an embedded minimal surface in $U$. 
\end{thm} 

In other words, under the given assumptions, a compact minimal surface whose boundary lies on the zero set of $f$ cannot penetrate an unbounded component of $\{f \ne 0\}$.\footnote{Some related results ruling out penetrating marginally outer trapped surfaces, which are closed (compact without boundary) and bounding, in static and stationary spacetimes, are obtained in \cite{Mars08} by different methods.}

The proof makes use of two lemmas which we present first.  Suppose $f$ is a static potential on 
$(M, g, \gamma, \rho)$.  
By Proposition \ref{prop-AF-static-f} (i), 
$ \lim_{ x \rightarrow \infty } | \nabla f |_g $ exists and is finite at each end. 
Therefore, 
\be \label{eq-gradf-bd}
 \Lambda  < \infty , \ \ \mathrm{where} \ \Lambda =   \sup_M | \nabla f |_g .
\ee
In the following lemmas,  we consider properties of the conformally deformed {\it Fermat} metric 
$$ \tilde g = f^{-2} g $$
on the open set $ \{ f \neq 0 \}$. 

\begin{lma} \label{lma-length-bd-1}
Let $U$ be a connected component of $\{ f \neq 0 \}$. 
Let $ p, q \in U$ be two distinct points. Let   $\dist_g (p, q) $ be the distance between $p $ and $ q$ in $(M, g)$. 
Given any curve $ \beta $ in $U$ that connects $p$ and $q$, let $ l(\beta, \tilde g) $ denote 
 the $\tilde g$-length of $\beta$. Then
  \be \label{eq-length-bd-1}
   l (\beta , \tilde g) \ge \frac{1}{ \Lambda } \ln \lf(   1 + \frac{\Lambda  \dist_g (p, q)    }{ \min \{ | f (p) |, | f (q) | \}     } \ri)  .
   \ee
\end{lma}

\begin{proof}
Without loss of  generality, we may assume $ f > 0 $ on $ U$.
Let $ l $ be the $g$-length of $ \beta$. 
We parametrize $\beta$ such that 
\be \label{eq-g-normalize}
 \beta (0) = p , \ \ \beta( l ) = q, \ \   |\beta'(t) |_g = 1 ,  \ \forall \ t \in [0, l] . 
 \ee
The $\tilde{g}$-length of $\beta$ is then  given by
\be \label{eq-length-est-1}
\begin{split}
l(\beta, \tilde g)  
%= & \ \int_0^L | \beta'(t) |_{\tilde{g}} d t \\
= & \  \int_0^l  \frac{1}{ f (\beta (t) )  } d t \\
\ge & \ \int_0^{d }   \frac{1}{ f (\beta (t) ) } d t ,
\end{split}
\ee
where $d =  \dist_g (p,q)  > 0  $. 
By \eqref{eq-g-normalize}, 
we have  
\be
\lf| \frac{d }{ dt}  f ( \beta(t) ) \ri| = \lf|  \la \nabla f , \beta'(t) \ra_g  \ri| \le \Lambda  .
\ee
Therefore, $ \forall \ t \in [ 0, l ]$, 
\be
 f ( \beta (t) )   - f ( \beta (0) )  \le    \Lambda  t .
\ee
This, together with the fact $ f (\beta (t) ) > 0 $,  shows 
\be
\frac{1}{f (\beta (t)  )} \ge \frac{1}{ f ( \beta (0) ) + \Lambda  t} . 
\ee
Hence, by \eqref{eq-length-est-1}, 
\be \label{eq-est-p}
\begin{split}
l ( \beta, \tilde g)  \ge & \ \int_0^{ d }  \frac{1}{ f ( \beta (0) )  + \Lambda  t}  d t \\
= & \ \frac{1}{ \Lambda }  \lf[ \ln {\lf( f ( p  ) + \Lambda  d \ri)  - \ln f ( p ) )  } \ri]  \\
= & \ \frac{1}{ \Lambda } \ln \lf(   1 + \frac{\Lambda  d }{ f (p)   }  \ri) \,,
\end{split}
\ee
Reversing the direction of $\beta$, we have 
\be \label{eq-est-q}
l (\beta, \tilde g) \ge \frac{1}{ \Lambda } \ln \lf(   1 + \frac{\Lambda  d }{ f (q)   }  \ri).
\ee
Therefore, \eqref{eq-length-bd-1} follows from  \eqref{eq-est-p} and \eqref{eq-est-q}. 
\end{proof}

\begin{lma} \label{lma-comp}
Let $U$ be a connected component of $ \{ f \neq 0 \}$.  The  metric  
$\tilde g $ is complete on $U$ 
\end{lma}

\begin{proof}
Let $ \beta: [0, T) \rightarrow (U , \tilde{g}) $ be an inextendible geodesic ray in $(U, \tilde{g})$. 
We want to show $ T = \infty$.

Suppose $ T < \infty$. Without loss of  generality, we assume $ | \beta'(t) |_{\tilde{g}} = 1 $, $ \forall \  t \in [0, T)$. 
Replacing $f$ by $ - f$, we may also assume $ f > 0 $ on $U$. 
Then $ | \beta'(t) |_{g} = f(\beta(t))$ and 
\be \label{eq-est-fbt}
\lf| \frac{d}{dt} f (\beta (t) ) \ri| =  \lf|  \la \nabla f , \beta'(t) \ra_g  \ri|  
\le  \Lambda  f (\beta(t) ) . 
\ee
By  \eqref{eq-est-fbt}, we have 
\be
\lf| \frac{d }{d t} \ln { f (\beta (t) ) } \ri| \le \Lambda ,
\ee
which implies 
\be \label{eq-est-fbt-2}
{ f (\beta (t) ) }  \le f (\beta (0) )  e^{ \Lambda t} , \ \ \forall \ t \in [0, T). 
\ee
Therefore, 
the  length $l$ of $\beta $ in $(M, g)$ satisfies 
\be \label{eq-length-g}
l   = \int_0^T f ( \beta (t) ) d t < \infty. 
\ee
Since $(M, g)$ is complete, \eqref{eq-length-g} implies  
\bee \label{eq-limit-beta}
 \lim_{ t \rightarrow T -  } \beta (t) = q 
 \eee
for some point $ q \in \overline{U} \subset M$. 

If $ q \in U $,   the geodesic $\beta : [0 , T) \rightarrow (U, \tilde{g})$
can  be extended beyond $T$, contradicting the inextendibility of $\beta$.

If  $ q \in \p U $, then 
$ f (q) = 0 $.  
By  Lemma \ref{lma-length-bd-1}, this implies 
the length of $\beta$ in $(U, \tilde g )$ is $\infty$, 
which contradicts  the assumption $ T < \infty$. 

Therefore, we must have  $T=\infty$, which shows $\tilde g$ is complete on $U$. 
\end{proof}

A  proof of completeness of the Fermat metric under somewhat different circumstances  has been considered in \cite[Lemma 3]{Mars13}.

\begin{proof}[Proof of Theorem \ref{non-minimal}]
Without losing generality, we assume $ f > 0 $ in $U$.
Let $ E_1, \ldots, E_k $ be all  the ends of $(M, g)$. For each large $r$, let 
$S_r^{(i)}$ be the coordinate sphere $\{ | x | = r \}$  near infinity on the end $E_i$.  
Define $ S_r = \cup_{i=1}^k S_r^{(i)}$ and 
$ S_{r, U} = S_r \cap U . $ 
Since $ U$ is unbounded, $ S_{r, U} \neq \emptyset$. 

Given any  compact subset  $S$  of $M$ such that $ S \setminus \p U $ is an embedded 
surface in $U$,  define $ S_{U} = S \setminus  \partial U $. 
Since $ S$ is compact, $ S \subset \Omega_r$ for sufficiently  large $r$ where $ \Omega_r$ is the 
bounded open set enclosed by $ S_r$ in $(M,g)$. 
Now consider the  two  disjoint surfaces $ S_U$ and $S_{r, U}$ in $ (U, \tilde{g})$.
Let $ \dist_{\tilde g} (\cdot, \cdot)  $ denote the distance functional  on $(U, \tilde g)$. 
We claim   
$ \dist_{\tilde g} (S_U, S_{r, U} ) > 0 $
and    there exists $ p  \in S_U$ and $q  \in S_{r, U}$ 
such that $ \dist_{\tilde g} (p , q ) = \dist_{\tilde g} (S_U, S_{r, U})  $. 

To prove the above claim, suppose $ \{ p_k \}  \subset S_U$, $ \{ q_k \} \subset S_{r,U}$ are 
 sequences of points such that 
\be \label{eq-limit-distance}
 \lim_{k \rightarrow \infty} \dist_{\tilde g} (p_k , q_k) = \dist_{\tilde g}(S_U, S_{r, U} ) . 
 \ee
By Lemma \ref{lma-length-bd-1}, we have
\be \label{eq-length-bd-2}
\dist_{\tilde g} (p_k, q_k) \ge  
 \frac{1}{ \Lambda } \ln \lf(   1 + \frac{  \Lambda  \dist_g ( S, S_r)      }{ \min \{  |f (p_k) |, | f (q_k) | \}     } \ri) ,
\ee
where $ \dist_g (S, S_r ) > 0 $ is  the distance between $S$ and $S_r$ in $(M, g)$. 
Since $ S $ and $ S_r$ are compact, there exists $p  \in S $ and $q \in S_{r}$ such that,
passing to a subsequence,  
$ \lim_{k \rightarrow  \infty} p_k = p $ {and}  
$\lim_{k \rightarrow \infty} q_k =  q $.
If either $ f( p) = 0 $ or $  f(q) = 0 $, then \eqref{eq-length-bd-2} 
implies 
$
\lim_{k \rightarrow \infty } \dist_{\tilde g} (p_k , q_k)  = \infty, 
$
 contradicting \eqref{eq-limit-distance} and the fact $ \dist_{\tilde g}(S_U, S_{r, U} ) < \infty$.
Therefore, we must have $ p \in S_U$ and $ q \in S_{r, U}$. 
Hence,
\be \dist_{\tilde g} (p, q) = \dist_{\tilde g}(S_U, S_{r, U}) .
\ee
Since $ S_U \cap S_{r, U}  =  \emptyset $, we also  have     $ \dist_{\tilde g}(S_U, S_{r, U}) > 0 $. 

To proceed, we apply Lemma \ref{lma-comp} to conclude that there exists a unit speed $\tilde g$-geodesic  
$ \beta: [0, L] \rightarrow (U, \tilde g)$ such that 
\be \label{eq-beta-m}
\beta (0) = p, \ \beta(L) = q   \ \mathrm{and} \ L = \dist_{\tilde g} (S_U, S_{r, U}) . 
\ee
Since $\beta$ minimizes the $\tilde g$-distance between points on $S_U$ and $S_{r,U}$, 
there are no $\tilde g$-cut points to $S_U$ along $\beta$, except possibly at the end point $q = \beta (L)$. 
Moreover,  by the fact  $ S_U  \subset \Omega_r $, 
we  have 
$ \beta([0, L) ) \subset \Omega_r $.
Hence, $ \tilde \mu : = \beta'(L)$ is $\tilde g$-normal to $S_r$   and is outward pointing (with respect to $\Omega_r$).  
As a result, $ \mu : = \frac{1}{f(q)}  \tilde \mu$ is the outward unit normal vector to $ S_r$ at $q$ in $(M, g)$. 
Therefore,  by the fact $ (M, g)$ is asymptotically flat, we have 
\be \label{eq-positive-H}
H (S_r, q) > 0 
\ee
for large $r$, where $ H(S_r, q)$ is the mean curvature of $S_r$ with respect to $ \mu$  in $(M, g)$. 
Our sign convention on the mean curvature is that a Euclidean ball has positive mean curvature 
with respect to its outward normal.)

Next, suppose $ S \setminus \p U $ is a minimal surface in $(M, g)$.
We will derive a contradiction to \eqref{eq-positive-H} using Equations \eqref{ein3}, \eqref{ein4} and the maximum principle. 
To illustrate the main idea used,  we first consider the case that  $q = \beta (L)$ is not a $\tilde g$-cut point to $S_U$ along $\beta$. 
In this case, on the surface $S_U$, there exists a small open neighborhood $W$ of $p$  such that  the map 
$$  \Phi (t, x) : =  \widetilde  \exp_x ( t \tilde \nu )  ,$$ 
where $\widetilde \exp_{ ( \cdot)  ( \cdot) } $ is the $\tilde g$-exponential map,
$ t \in [0, L] $, $ x \in W$ and  $ \tilde \nu $ is a $\tilde g$-unit  vector field normal to $W$ with  $ \tilde \nu (p) = \beta'(0) $, 
is a diffeomorphism from $ [0, L] \times W $ onto its image in $U$. 
For each $ t \in [0, L]$, let $ W_t =  \Phi (t, W)$ and 
%($W_t$ is simply  the surface  obtained  by pushing $W$ along its normal geodesics 
%in the direction of $\tilde n$ a distance $t$ in $(U, \tilde g)$.)
 let $H = H(t)$ be the mean curvature of $W_t$ with respect to $ \nu (t) = f^{-1} \Phi_* ( \frac{\p }{\p t} )$ in $(M, g)$. 
In what follows, we perform all the computations with respect to the original  metric  $g$.  
 (The metric  $\tilde g$  was only used   to produce the variation $t  \mapsto W_t$ in $U$.)

We now obtain a monotonicity formula involving $H = H(t)$ (along the lines of that  considered in \cite{Gal84}; see also \cite{Mars13, Lee-Neves13}).
By a standard computation, we have 
\beq\label{eq-Hder}
\frac{\d H}{\d t} = -\L_{t} f - (\ric( \nu , \nu ) +  |\Pi|^2)f \,,
\eeq
where $\L_t$ is the Laplacian on $W_t$ and $\Pi = \Pi (t)$ is the second fundamental form of $W_t$. 
%, and $\nu$ has been extended to be the unit normal field to the $W_t$'s.
%\footnote{I am using a sign convention for the mean curvature opposite that in \cite{Gal84}, namely, $H = +\, \div\, \nu$.}  
The Laplacians $\L$ on $M$ and $\L_t$ on $W_t$ are related by
\beq
\L f = \L_t f + \D^2 f(\nu,\nu) + \frac{H}{f} \frac{\d f}{\d t}  \,.
\eeq
Substituting  (\ref{ein3}) and (\ref{ein4}) into the above gives
\beq \label{eq-plug}
\L_t f = \lf[ \rho + \gamma(\nu, \nu) - \ric(\nu,\nu) \ri] f - \frac{H}{f} \frac{\d f}{\d t} \, .
\eeq
It follows from \eqref{eq-Hder} and \eqref{eq-plug} that 
\beq
\frac{\d H}{\d t} =  - \lf[ |\Pi|^2 + \gamma (\nu, \nu) + \rho \ri]  f + \frac{H}{f} \frac{\d f}{\d t} \,
\eeq
or equivalently 
\beq 
\frac{\d}{\d t} \lf( \frac{H}{f} \ri) = -  \lf[ |\Pi|^2 + \gamma (\nu, \nu) + \rho \ri]   .
\eeq
Hence by the NEC  \eqref{eq-NE}, we have 
\be \label{eq-mono}
\frac{\d}{\d t} \lf( \frac{H}{f}  \ri) \le 0 \, . 
\ee
 By \eqref{eq-mono}, if $W = W_0 $ has mean curvature $H (0) \le 0$ with respect to $\nu$, then 
 $ W_L$ has mean curvature $H_L \le 0 $.   
 On the other hand,  by the minimizing property of $\beta$, i.e. \eqref{eq-beta-m}, 
 $W_{L} \subset \bar{\Omega}_r $ and $W_L$  touches $S_r$ at $q$.  
Therefore, the inequality $ H_L \le 0 $ contradicts \eqref{eq-positive-H} and the maximum principle. 
Hence, $ S_U$ can not be a minimal surface in $(M, g)$. 

To complete the proof, we need to handle the case $ q = \beta (L)$ is a $\tilde g$-cut point to $S_U$ along $\beta$.
We will reduce this case to the case just considered above by the following procedure. 
Let $ \tilde \eta$ denote the inward, unit normal to $S_{r, U}$ in $(U, g)$. In particular, $ \tilde \eta (q) = - \beta'(L) $. 
%Let $ \exp_y ( \cdot ) $ be the exponential map at $ y \in S_r$ in $(M, g)$.
Let $D \subset S_{r, U} $ be a small open neighborhood 
 of $q$. There exists a small $ \epsilon > 0$ such that the map
 \be
 \Psi ( t, y) : =  \tilde \exp_y ( t \tilde  \eta ) ,
 \ee
 where $ t \in [0, \epsilon ]$ and $ y \in D$, 
 is a diffeomorphism from $ [0, \epsilon] \times D$ onto its image in $U$. 
 Define $D_\epsilon = \Psi (\epsilon , D)$. 
 Clearly, $ \hat q = \beta (L - \epsilon ) \in D_\epsilon $. Moreover, 
 when restricted to $[0 , L - \epsilon] $, $\beta$  minimizes distance between $D_\epsilon $ and $ S_U$ in $(U, g)$. 
Using \eqref{eq-positive-H},  by choosing $ \epsilon $ sufficiently small, we  further have 
\be
H(D_\epsilon, \hat q) > 0 , 
\ee
where $H(D_\epsilon, \hat q) $ is the mean curvature of $D_\epsilon $ at $ \hat q$ with respect to the outward normal in $(M, g)$. 
Since $ \hat q = \beta (L- \epsilon)$ is not a $\tilde g$-cut point to $S_U$ along $\beta$, 
therefore repeating the previous argument with $q$ replaced by $ \hat q$ and $S_{r, U}$ replaced by $ D_\epsilon$, we again 
obtain a contradiction to the assumption that $ S_U$ is minimal in $(M, g)$. 
This completes the proof. 
\end{proof}

\begin{remark}
Among results from 
 Section \ref{sec-p}, we have only used  Proposition \ref{prop-AF-static-f} (i) 
 to obtain  the gradient estimate \eqref{eq-gradf-bd}
 in the proof of Theorem \ref{non-minimal}.
\end{remark}

\begin{remark}
 Theorem  \ref{non-minimal} does not depend on the dimension of $M$, i.e., it holds in all dimensions $\ge 3$.
Indeed, while the proof makes use of Proposition~\ref{prop-AF-static-f}, it can be shown that this
proposition extends to higher dimensions. 
%{\color{red} We choose to focus on  the $3$-dimension in this section only because 
%the rigidity results in the next section are  three dimensional.}
%\mnote{I do't think this sentence is really necessaary}
\end{remark}

\begin{remark}
Even though in Theorem \ref{non-minimal} 
we focus on asymptotically flat manifolds  which are complete without boundary, 
it is clear from its proof that the conclusion of Theorem \ref{non-minimal} 
holds for those  $(M, g)$ which have nonempty boundary $\p M$ (possibly disconnected)
 on which $f$ vanishes.  
\end{remark}

\section{Rigidity properties of static potentials}

We now establish some rigidity properties of a noncompact, connected  
component (assuming such exists) of the zero set of a static potential  $f$; cf. Theorem \ref{thm-stab}.  
 As a corollary to this, we consider circumstances which imply global rigidity.

We continue to assume $(M, g)$ is a complete,  asymptotically flat $3$-manifold without boundary, with finitely many ends.  The triple ($g$, $\gamma$, $ \rho $) is assumed to  satisfy the decay assumptions \eqref{eq-AF-def} and \eqref{eq-gamma-decay} 
on  each end.   

The following is a fundamental (and almost immediate) consequence of Theorem \ref{non-minimal}.

\begin{thm} \label{thm-minimizing}
Assume $(M, g, \gamma, \rho)$ satisfies the NEC and $f$ is a  static potential. 
If $\Sigma$ is a  noncompact, connected component of 
$ f^{-1}(0)$, then $\Sigma$   is   a strictly area minimizing surface in $(M, g)$. 
\end{thm}

\begin{proof}
By Lemma \ref{lma-static-basic} (i), $ \nabla f$ 
is a nowhere vanishing normal to $\Sigma$. 
Hence, there exists a connected open neighborhood $W$ of $\Sigma$ such that
$ W \setminus  \Sigma$ is the disjoint union of two connected, unbounded, open sets $W_+$ and $W_-$ satisfying 
$ W_+ \subset \{ f > 0 \}$ and $ W_- \subset \{ f < 0 \}$. 
Let $ U_+$ and  $U_-$  be  the connected component  of $\{ f > 0 \}$ and $\{ f < 0 \}$  that 
contains $W_+$ and  $ W_-$ respectively. 
As $ W_+$, $ W_-$ are unbounded, so are $U_+$ and $U_-$. 

To show that $ \Sigma$ is strictly area minimizing, we make use of solutions to the Plateau problem in $(M, g)$ (cf. \cite{Almgren-Simon-79, Meeks-Yau-82}).
Let $ \{ D_k \}_{k=1}^\infty $ be an exhaustion sequence of $\Sigma$ such that each $ D_k$ is a bounded domain in $\Sigma$ with smooth boundary  $\Gamma_k$. (For instance,  $D_k$ may be taken as the subset of $\Sigma$ enclosed by the union of  large coordinates spheres  near infinity on each end of $(M, g)$ in which $\Sigma$ extends to the infinity. 
By  Prop. \ref{prop-AF-static-zeroset},  each such sphere intersects $\Sigma$ transversely.)
%Alternatively, $D_k $ may be taken as the set $D_{r_k}$ used in the proof of Theorem \ref{thm-stab} below 
%with $ r_k  \rightarrow \infty$.)
%, with each such sphere perturbed to be  transverse to $\Sigma$.)
 Since $(M, g)$ is asymptotically flat 
 and hence foliated by mean convex spheres near infinity at each end, 
there exists a compact, embedded, minimal surface $S_k $ in $(M, g)$ such that $ \p S _k= \Gamma_k$ and 
$ S_k$ minimizes area among all compact surfaces having  boundary $\Gamma_k$.   
By Theorem \ref{non-minimal}, $ S_k \cap U_+ = \emptyset = S_k \cap U_- $.  Therefore, 
$ S_k \subset \Sigma$ and  hence $S_k= D_k $. This shows that  $D_k$ strictly minimizes area among all surfaces with the same boundary.
Since $ \{ D_k \}$ exhausts $\Sigma$, we conclude that $ \Sigma$ is strictly area minimizing.  
\end{proof}

So far we have only assumed the null energy condition. In what follows, we impose the stronger energy condition, 
\be\label{eq-EC}
\rho \ge | \gamma (X, X)| \ \ \mathrm{for \ all \ unit \ vectors} \ X .
\ee
Note that if $f$ is a  static potential, then on the region $\{f \ne 0 \}$ (whose closure is all of $M$)  
this energy inequality is a consequence of the spacetime dominant energy condition as applied to the static metric \eqref{eq-static}.  Moreover, inequality \eqref{eq-EC} implies
$ R \ge 0 $,  where $R$ is the scalar curvature of $(M, g)$.  Indeed, the trace of \eqref{ein3} and \eqref{ein4} imply that 
\beq\label{eq-rho}
\rho = \frac12 R 
\eeq
on the set $\{f \ne 0 \}$, and hence, by continuity, on all of $M$.

\begin{thm} \label{thm-stab}
Assume $(M, g, \gamma, \rho)$ satisfies the energy condition \eqref{eq-EC} and $f$ is a  static potential. 
Suppose  $\Sigma$ is a  noncompact, connected component of 
$ f^{-1}(0)$.  Then $\Sigma$   is   a  strictly area minimizing, totally geodesic surface 
that is isometric to the Euclidean plane $ \R^2$.
Moreover,  $(M, g)$ is flat  along $ \Sigma$ and $ \rho =0 $, $\gamma = 0 $ along $ \Sigma$.  
\end{thm}

\begin{proof}
Since \eqref{eq-EC} implies NEC, $\Sigma$ is strictly area minimizing by Theorem~\ref{thm-minimizing}. 
In particular, $ \Sigma$ is stable, i.e.
\be \label{eq-stable-1}
\int_\Sigma  \lf[  | \nabla_\Sigma \eta |^2 - ( \Ric (\nu, \nu) + | \Pi |^2 ) \eta^2  \ri] d \sigma \ge 0 .
\ee
Here  $ \eta$ is  any Lipschitz function on $ \Sigma$ with compact support, 
 $ \nabla_\Sigma$ and $ d \sigma$ denote the gradient and  the area form on $\Sigma$ respectively,  
$ \nu$ is a unit normal along $\Sigma$, and $ \Pi$ is the second fundamental form of $\Sigma$. 
 By Lemma \ref{lma-static-basic} (i), $ \Pi = 0 $.
 It follows from  the Gauss equation that 
\be \label{eq-Gauss}
2 K = R - 2 \Ric (\nu, \nu) , 
\ee 
where $K$ is the Gaussian curvature of $\Sigma$.
Hence,  \eqref{eq-stable-1}  becomes  
\be \label{eq-stable-2}
\int_\Sigma  \lf[  | \nabla_\Sigma \eta |^2 -  \lf( \frac12 R - K \ri) \eta^2  \ri] d \sigma \ge 0 .
\ee

Now  let $ E_1, \ldots, E_k $ be those ends of $(M, g)$ such that $ \Sigma $ extends to the infinity of $E_i$,  $ 1 \le i \le k$.
By Proposition \ref{prop-AF-static-zeroset},  near the infinity of each $E_i$,   $\Sigma = f^{-1}(0)$ which is 
 the graph of some function $ q = q(y_2, y_3)$ satisfying \eqref{eq-condition-q} in a coordinate  chart $\{ y_1, y_2, y_3 \}$. 
For each large $ r $, let $ \gamma_r^{(i)} \subset \Sigma $ be the curve that  is the graph of $q$ over the circle 
$\{ y_2^2 + y_3^2  = r^2 \}$ in the $y_2 y_3$-plane.  
Let $ \Gamma_r = \cup_{i=1}^k \gamma_r^{(i)}$ and $ D_r $ be the bounded region in $ \Sigma $  enclosed by 
$ \Gamma_r$. 
%Let $ \mathrm{Area}(D_r)$ denote the area of $ D_r$.  
We claim
\be \label{eq-area-growth}
\mathrm{Area} (D_r) \le C_1 r^2 
\ee
for some constant $ C_1 > 0$. 
To see this, 
%working in the $\{ y_i \}$ coordinate chart,  
one can   consider the map 
$  F: \Omega_C \rightarrow \Sigma $  given   by $ F  (y_2, y_3) = ( q (y_2, y_3) , y_2, y_3 ) $, 
where  $ \Omega_C$ is  the exterior  region
%to the circle $\{ y_2^2 + y_3^2 = C^2 \}$
 in the $y_2 y_3$-plane defined   in Proposition  \ref{prop-AF-static-zeroset}. 
Let  $ \sigma = F^* ( g) $ be  the pulled back  metric  from $ \Sigma$ to $ \Omega_C$
and  let 
$ \sigma_{\alpha \beta} = \sigma (\p_{y_\alpha}, \p_{y_\beta} ) $ where 
 $ \alpha, \beta \in \{ 2, 3 \}$.
It follows from \eqref{eq-AF-def} and  \eqref{eq-condition-q}
that
\be \label{eq-decay-sigma}
\sigma_{\alpha \beta} = \delta_{\alpha \beta} + h_{\alpha \beta}, \ 
\ee
where 
$ h_{\alpha \beta} $ satisfies
\be \label{eq-decay-tau}
| h_{\alpha \beta} | + | \bar{y} | | \p  h_{\alpha \beta} | = O ( | \bar{y} |^{-\tau} )
\ee
with $ | \bar{y} | = \sqrt{ y_2^2 + y_3^2} $.  It is readily seen that 
 \eqref{eq-decay-sigma} and \eqref{eq-decay-tau} imply  \eqref{eq-area-growth}.    
 
Next,  we apply a logarithmic cut-off argument using \eqref{eq-stable-2} and \eqref{eq-area-growth}.  
Given any large integer $m$, define  $ \eta_m  $ on $ \Sigma$ 
by 
\be \label{eq-def-etam}
 \eta_m (p) = 
\left\{ 
\begin{array}{cc}
1,  & \ p \in  D_{e^m} \\
2 - \frac{\log | \bar{y} (p)| }{m}, &  e^m \le | \bar{y} (p) |   \le e^{2m} \\
0 & \ p \notin D_{2 e^m}  .
\end{array}
\right.
\ee 
Plugging  $ \eta = \eta_m$ in \eqref{eq-stable-2} gives 
\be \label{eq-log-cut}
\begin{split}
 \int_\Sigma \lf( \frac12 R - K \ri) \eta_m^2  d \sigma 
\le & \ \frac{1}{m^2} \int_{ D_{e^{2m} } \setminus D_{e^m}  }  \frac{ | \nabla_\Sigma | \bar{y} |  |^2 }{| \bar{y} |^2} d \sigma \\
\le  & \  \frac{1}{m^2}   \sum_{l=m+1}^{2m} \int_{ D_{e^{l} } (p) \setminus D_{e^{l-1} } (p) }  \frac{C_2}{| \bar{y} |^2 } d \sigma  \\
\le & \ \frac{1}{m^2} C_1  C_2 \sum_{l=m+1}^{2m}  e^{ - 2(l-1)}   e^{2l} =  \frac{1}{m} C_1 C_2  e^2 ,
\end{split}
\ee
where we have used \eqref{eq-area-growth} and the fact 
$ | \nabla_\Sigma | \bar{y} | |  \le | \nabla | \bar{y} | | \le C_2$
for some constant $ C_2 > 0 $ when $ \bar{y}$ is large. 
To proceed, we note that
\be
R = O ( | y |^{-\tau- 2} ) = O( | \bar{y} |^{-\tau - 2}) 
\ee
by \eqref{eq-AF-def}, and 
\be
K = O ( | y |^{-\tau- 2} ) = O( | \bar{y} |^{-\tau - 2}) 
\ee
by \eqref{eq-AF-def} and \eqref{eq-Gauss}. These together with  \eqref{eq-decay-sigma} and  \eqref{eq-decay-tau}
imply
\be
\int_\Sigma | R | d \sigma < \infty \  \ \mathrm{and} \  \ \int_\Sigma | K | d \sigma < \infty . 
\ee
Thus,  letting $m \rightarrow \infty$ in \eqref{eq-log-cut}, we have  
\be \label{eq-int-K-ge}
\int_\Sigma K d \sigma  \ge  \frac12 \int_\Sigma R d \sigma  . 
\ee
On the other hand,  by   the Gauss-Bonnet theorem and  \eqref{eq-g-curvature} in  Proposition \ref{prop-AF-static-zeroset}, 
\be \label{eq-total-K-alpha}
\begin{split}
\int_{\Sigma} K d \sigma = & \ \lim_{r \rightarrow \infty}  \lf(  2\pi \chi (D_r)  -  \sum_{i =1}^k   \int_{\gamma^{(i)}_{r} } \kappa \ d s \ri) \\
= & \    2\pi  ( \chi (\Sigma )  - k )    \\
\le & \ 0 
\end{split}
\ee
since  $k \ge 1 $ and   the  Euler characteristic $\chi (\cdot )$ of any connected,  noncompact  surface is at most $1$. 
(So far we have not imposed the energy condition \eqref{eq-EC}; thus \eqref{eq-int-K-ge} and  \eqref{eq-total-K-alpha} hold  without a sign assumption on $ R$.) 
Now applying the  fact $ R \ge 0 $ (which is by \eqref{eq-EC} and \eqref{eq-rho}),
we  conclude from \eqref{eq-int-K-ge} and   \eqref{eq-total-K-alpha}  that 
\be
\int_\Sigma K d \sigma = 0, \  k = 1, \ \chi (\Sigma) = 1, \ \mathrm{and} \ R = 0 \ \mathrm{along} \ \Sigma .  
\ee

Finally, we want to show  $ K = 0 $ along $ \Sigma$. 
We follow  an argument of  Fischer-Colbrie  and Schoen in \cite{FCS}.  
Note that $ R = 0 $ implies 
$ K = - \Ric (\nu, \nu) $. 
Hence, the stability operator on $\Sigma$ becomes 
$ L  = \Delta_\Sigma  - K  $,
where $ \Delta_\Sigma$ is the Laplacian on $\Sigma$. 
Since $ \Sigma $ is complete,
by \cite[Theorem 1]{FCS}, there exists a positive function $ v $ on $\Sigma$ 
satisfying 
\be \label{eq-def-v}
\Delta_\Sigma v - K v = 0 . 
\ee
Define $ w = \log {v} $, then
\be \label{eq-w}
\Delta_\Sigma w + | \nabla_\Sigma w |^2 = K . 
\ee
Let $ \eta_m $ be given as in \eqref{eq-def-etam}.
Multiply \eqref{eq-w} by $\eta_m^2$ and integrate by parts,  
\be
\begin{split}
\int_\Sigma K \eta_m^2 d \sigma = & \ 
- 2 \int_\Sigma  \eta_m \la \nabla_\Sigma \eta_m , \nabla_\Sigma w \ra d \sigma 
+ \int_\Sigma \eta_m^2 | \nabla_\Sigma w|^2  d \sigma  \\
\ge & \ - \int_\Sigma  \lf( \frac12  \eta_m^2 | \nabla_\Sigma w|^2  + 2 | \nabla_\Sigma \eta_m |^2   \ri)  d \sigma  
+ \int_\Sigma \eta_m^2 | \nabla_\Sigma w|^2  d \sigma  \\
=  & \ - 2 \int_\Sigma  | \nabla_\Sigma \eta_m |^2     d \sigma  
+ \frac12 \int_\Sigma \eta_m^2 | \nabla_\Sigma w|^2  d \sigma  .
\end{split}
\ee
Let $ m \rightarrow \infty$ and  use  the fact $ \int_\Sigma | \nabla_\Sigma \eta_m |^2 d \sigma = O ( m^{-1} )$
(cf. \eqref{eq-log-cut}), we have 
\be
\int_\Sigma K d \sigma \ge \frac12 \int_\Sigma | \nabla_\Sigma w |^2 d \sigma .
\ee
Since  $ \int_\Sigma K d \sigma = 0 $,  this  proves   $  \nabla_\Sigma w = 0 $ on $ \Sigma$, hence 
$ v $ is a positive constant. From \eqref{eq-def-v}, we conclude 
$K = 0 $ along $ \Sigma$.  
As a result,    $ \Sigma$ is isometrically covered by  $ \R^2$.
However, $ \Sigma$ cannot be a cylinder since $ \chi (\Sigma) = 1$. Hence, 
$\Sigma$ is isometric to $ \R^2$. 

To complete the proof, 
we note that  $ R = 0 $ along $ \Sigma$ implies $ \rho = 0 $ and $ \gamma = 0 $ along $\Sigma$ by  \eqref{eq-EC} and \eqref{eq-rho}. 
Therefore,  by Lemma \ref{lma-static-basic} (iii)  and \eqref{eq-Ric-cross}, $ (M,g)$ is flat along $ \Sigma$. 
\end{proof} 

As an application of Theorem \ref{thm-stab},  we have

\begin{cor}\label{cor-f-zerofree}
Assume $(M, g, \gamma, \rho)$ satisfies the energy condition \eqref{eq-EC} and $f$ is a  static potential. 
Suppose that, outside a compact set,  either $\rho > 0 $ or  $ (M, g)$ is non-flat at each point.
Then  $f$ must be bounded on $M$  and hence $ f > 0 $ or $ f < 0 $ 
near infinity on each end of $(M, g)$. 
\end{cor}

\begin{proof}
If $ f $ is unbounded on $M$,   
 Proposition \ref{prop-AF-static-zeroset} implies that 
$f^{-1}(0)$ must have a noncompact,  connected component $\Sigma$.
By Theorem \ref{thm-stab}, $(M, g)$ is flat and $ (\rho, \gamma) = (0, 0)$ along $ \Sigma$,
which contradicts the assumption that $(M, g)$ is non-flat or $\rho > 0 $ outside a compact set.
 Therefore, $f$ must be bounded on $M$. 
By Proposition \ref{prop-AF-static-f} (ii), $ f $ is either positive or negative near infinity on each end of $(M, g)$. 
\end{proof}

%{\color{red} 
%\begin{remark}
%Inserting a remark corresponding to the last sentence in the abstract?\mnote{If you want to add something here it is OK with me.  I admit we could probably say more. But, I think the comments towards the end of introduction are sufficient.}
%\end{remark}
%}

Together with the positive mass theorem (\cite{SchoenYau79, Witten81}), Corollary \ref{cor-f-zerofree}
implies the following rigidity result for asymptotically Schwarzschildean manifolds that admit an unbounded 
static potential. 

\begin{cor}\label{cor-AS}
Assume $(M, g, \gamma, \rho)$ satisfies the energy condition \eqref{eq-EC} and $f$ is a  static potential. 
Suppose $ (M, g)$ is asymptotically Schwarzschildean at each end, i.e.,
there exists a  coordinate chart $ \{ x_1, x_2, x_3 \}$   near infinity  on  each end $E_\alpha $,  
 in which the metric $g$ satisfies
\be \label{eq-AS}
g_{ij} = \lf( 1 + \frac{m_\alpha }{2 | x|} \ri)^4 \delta_{ij}+ p_{ij},
\ee
where     $ |p_{ij} | + |x| | \p p_{ij} | + |x|^2 | \p^2 p_{ij} | =O (|x|^{-2})$ 
and $ m_\alpha $ is some constant. Here $ 1 \le \alpha \le k $ and $E_1, \ldots, E_k $ denote  all the ends of $(M, g)$. 
Then, if $f$ is unbounded,  $(M, g)$ is isometric to $ \R^3$.
\end{cor}

\begin{proof}
On each $ E_\alpha$,  \eqref{eq-AS}  implies that the Ricci curvature of $g$  satisfies 
\be \label{eq-Ricci-AS}
\Ric(\p_{x_i}, \p_{x_j} )  (x) = \frac{m_\alpha }{ | x |^3 } 
\lf[ 1 + \frac{m_\alpha }{ 2 | x |} \ri]^{-2} \lf( \delta_{ij} - 3 \frac{ x_i x_j }{ | x |^2 } \ri) 
+ O ( | x |^{-4} ) .
\ee
(See  \cite[Lemma 1.2]{Huisken-Yau1996} for instance).
If  $ m_\alpha \neq 0 $, $\forall \ \alpha = 1, \ldots, k $, then 
$g$ is non-flat at each point outside a compact set in $M$.
Thus $f$ must be bounded on $M$ by Corollary \ref{cor-f-zerofree}. This  contradicts  the assumption
$f$ is unbounded.  Therefore, $ m_\alpha = 0 $ for some $ \alpha$. 

On the other hand, by \eqref{eq-AS} and  the fact $ R \ge 0 $,
 the ADM mass (\cite{ADM61}) of $g$ at the end  $E_\alpha$ 
exists and is equal to $ m_\alpha$.  
Hence  it follows from   the  positive mass theorem 
that $(M, g)$ is isometric to $ \R^3$. 
\end{proof}

\begin{remark}
Corollary \ref{cor-AS} may be compared with \cite[Theorem 1.1]{static-af}. 
Corollary \ref{cor-AS}
is a  global result for  a complete, boundaryless,  asymptotically Schwarzschildean manifold on which the 
static potential equations \eqref{ein1} and \eqref{ein2} allow a nontrivial pair $(\rho, \gamma)$.
 Theorem 1.1 in \cite{static-af}, proved  for $(\rho, \gamma) = (0, 0)$,  has a local feature that it is 
 applicable on an  asymptotically Schwarzschildean  end.
\end{remark}

\begin{remark}\label{conj}
We note that Corollary \ref{cor-AS} also follows from Theorem \ref{thm-minimizing} and a result of Carlotto 
in \cite{Carlotto}. In \cite[Theorem 1]{Carlotto}, Carlotto proved that if $(M, g)$ is an asymptotically Schwarzschildean 
$3$-manifold of nonnegative scalar curvature in which there exists a complete, noncompact,  properly embedded stable minimal surface, then $(M, g)$ is isometric to $\R^3$. 
The proof of  \cite[Theorem 1]{Carlotto} also makes  use of the positive mass theorem. 

Work of Carlotto and Schoen \cite{CarSchoen} shows in a dramatic fashion that Carlotto's result no longer holds if the metric is merely assumed to obey the asymptotic condition \eqref{eq-AF-def} for $\tau \in (\frac12, 1)$.   However, Schoen \cite{SchoenMSRI} has raised the question as to whether there can be an area minimizing (not just stable) asymptotically planar minimal surface in a nontrivial asymptotically flat  manifold (obeying this weaker decay) with nonnegative scalar curvature.  In view of Theorem \ref{thm-minimizing}, these considerations lead us to believe that Corollary \ref{cor-AS} remains valid under the weaker decay  \eqref{eq-AF-def}.  In the vacuum case ($\rho = 0, \gamma = 0$) a spacetime approach which implies such rigidity follows from \cite{Beig-Chrusciel-97, Beig-Chrusciel-97b} and references therein.
\end{remark}

\medspace
\noindent
{\it Final remark:}
%\mnote{remark added}  
This paper was submitted for publication to CAG on December 11,
%\mnote{This is probably not quite right; when did you submit?}
2014 and was accepted for publication on October 27, 2015.  As this paper now goes to press a new paper of Chodosh and Eichmair has appeared on the arXiv (cf.  arXiv:1510.07406) which directly relates to the question of Schoen commented upon in Remark \ref{conj}.   The authors prove that the only asymptotically flat 
Riemannian three-manifold (having the general asymptotics  \eqref{eq-AF-def}), with non-negative scalar curvature that admits a non-compact {\it area-minimizing \it boundary} is flat $\bbR^3$.  Given our understanding of their Theorem 1.2, it can be combined with our Theorem \ref{thm-stab} to conclude that Corollary 4.2 holds under the weaker asymptotics \eqref{eq-AF-def}, as conjectured in Remark \ref{conj}.

\medskip

\noindent\emph{Acknowledgements.} 
The work of GJG was partially supported by NSF grant DMS-1313724 and by a Fellows Program grant from
the Simons Foundation (Grant No.  63943).  GJG would like to thank the University of
Vienna Gravitational Physics group for its gracious hospitality, during which part
of the work on this paper was carried out.  
The work of PM was partially supported by a  Simons Foundation Collaboration Grant for Mathematicians \#281105.   We would like to thank Piotr Chru\'sciel and Brian White for many helpful comments.  We would also like to thank Alessandro Carlotto, Michael Eichmair and Luen-Fai Tam for their interest in this work.  In particular, we would like to thank Carlotto for a careful reading of an earlier version of the paper.

\providecommand{\bysame}{\leavevmode\hbox to3em{\hrulefill}\thinspace}
\providecommand{\MR}{\relax\ifhmode\unskip\space\fi MR }
\providecommand{\MRhref}[2]{%
  \href{http://www.ams.org/mathscinet-getitem?mr=#1}{#2}
}
\providecommand{\href}[2]{#2}

\end{document}